\documentclass[reqno]{amsart}
 \usepackage[backref,colorlinks,linkcolor=blue,anchorcolor=green,citecolor=blue]{hyperref}
 \usepackage{tikz}
 \usetikzlibrary{arrows,automata,backgrounds,calendar,mindmap,trees}
\usepackage{amsfonts,amssymb}
\usepackage{amsmath}
\usepackage{graphicx}
\usepackage{enumerate}
		 \newtheorem{theorem}{Theorem}

          \newtheorem{example}[theorem]{Example}
          \newtheorem{remark}[theorem]{Remark}
          \newtheorem{lemma}[theorem]{Lemma}
          \newtheorem{corollary}[theorem]{Corollary}

   \author{Bao Quoc Tang}
   \address{Institute for Mathematics and Scientific Computing,\\
   University of Graz, Austria}
   \email{quoc.tang@uni-graz.at}
   
    \title[Global classical solutions to reaction-diffusion systems]{Global classical solutions to reaction-diffusion systems in one and two dimensions}

\begin{document}

              \begin{abstract}
                   The global existence of classical solutions to reaction-diffusion systems in dimensions one and two is proved. The considered systems are assumed to satisfy an {\it entropy inequality} and have nonlinearities with at most cubic growth in 1D or at most quadratic growth in 2D. This global existence was already proved in [T. Goudon and A. Vasseur, Ann. Sci. \'Ecole Norm. Sup. (4) 43 (2010), no. 1, 117--142] by a De Giorgi method.  In this paper, we give a simplified proof by using a modified Gagliardo-Nirenberg inequality and the regularity of the heat operator. Moreover, the classical solution is proved to have $L^{\infty}$-norm growing at most polynomially in time. As an application, solutions to chemical reaction-diffusion systems satisfying the so-called complex balance condition are proved to converge exponentially to equilibrium in $L^{\infty}$-norm.
              \end{abstract}	
         
         \maketitle

\noindent
{\bf Classification AMS 2010}: 35K57, 35B40, 35Q92, 80A30, 80A32

\medskip

\noindent
{\bf Keywords:} Reaction-diffusion systems; Global classical solutions; Entropy estimates; Chemical reaction networks.

\section{Introduction}
\medskip
Let $\Omega \subset \mathbb R^d$ with $d=1,2$ be a bounded domain with smooth boundary $\partial\Omega$ in case $d = 2$ (e.g. $\partial\Omega$ is of class $C^{2+\epsilon}$ for $\epsilon>0$). In this paper, we study the global existence of classical soltuions to nonlinear reaction-diffusion system for $u = (u_1, \ldots, u_N): \Omega\times\mathbb R_+ \to \mathbb R^N$
\begin{equation}\label{cubic-system}
\begin{aligned}
\partial_t u_i - d_i\Delta u_i &= f_i(u), && x\in\Omega, \; t>0,\\
\nabla u_i\cdot \nu &= 0, && t>0,\\
u_i(x,0) &= u_{i,0}(x), &&x\in\Omega
\end{aligned}
\end{equation}
where $d_i>0$ for $i=1,\ldots, N$ are diffusion coefficients, $\nu$ is the outward normal on $\partial\Omega$, the nonlinearities $f_i: \mathbb R^N \to \mathbb R$ are assumed satisfy some properties which will be specified later.

\medskip
Reaction-diffusion systems (RD systems) of type \eqref{cubic-system} are used to model many phenomena in natural sciences such as physics, chemistry or biology. The global existence of (classical, strong, weak) solutions to systems of type \eqref{cubic-system} is thus of importance, and has been extensively studied in literature and become nowadays a classical topic (see e.g. \cite{Ama85,HRP87,Ladyzhenskaya,Mog89,Rot87,HLV98} and references therein). However, it still poses many open problems, since it is usually difficult to obtain suitable {\it a priori} estimates of solutions to general RD systems (maximum principle fails to apply to RD systems except very special cases). 

\medskip
This paper studies global existence of classical solutions to RD system \eqref{cubic-system} in one and two dimensions, where the nonlinearities $f_i: \mathbb R^N \to \mathbb R$ are assumed to be locally Lipschitz continuous and satisfy the following conditions:
\begin{itemize}
	\item (Positivity preserving) For all $u\in \mathbb R_+^N$ it holds
	\begin{equation}\label{positive-preserve}\tag{$\mathbf{P}$}
	f_i(u_1, \ldots, u_{i-1}, 0, u_{i+1}, \ldots, u_N) \geq 0 \quad \text{ for all } \quad i = 1, \ldots, N.
	\end{equation}
	\item (Entropy inequality) There exists $\mu_1, \ldots, \mu_N \in \mathbb R$, such that
	\begin{equation}\label{entropy-inequality}\tag{$\mathbf{E}$}
	\sum_{i=1}^{N}f_i(u)(\mu_i + \log u_i) \leq 0 \quad \text{ for all } \quad u\in \mathbb R^N_+.
	\end{equation}
	\item (Growth condition) For all $u\in \mathbb R^N$ it holds for all $i=1,\ldots, N$
	\begin{equation}\label{cubic-growth}\tag{$\mathbf{G}$}
	|f_i(u)| \leq K(|u|^{\mu} + 1),
	\end{equation}
	for some constant $K>0$, where the growth rate $\mu$ satisfies
	\[
	\mu \leq 3 \quad \text{ when } \quad d = 1,
	\]
	and 
	\[
	\mu \leq 2 \quad \text{ when } \quad d = 2.
	\]
	Here $|u| = \sum_{i=1}^{N}|u_i|$ for all $u\in \mathbb R^N$.
\end{itemize}
\begin{example}
	We give an example of a reaction-diffusion system satisfying \eqref{positive-preserve}, \eqref{entropy-inequality} and \eqref{cubic-growth}. Consider the following reversible chemical reaction involving sulfur dioxide, oxygen and sulfur trioxide
	\begin{equation*}
	2SO_2 + O_2 \leftrightharpoons 2SO_3
	\end{equation*}
	where the forward and backward reaction rate constants are assumed to be one. Denote by $u_1, u_2, u_3$ the concentrations of $SO_2$, $O_2$ and $SO_3$ respectively. The corresponding reaction-diffusion system for $u = (u_1, u_2, u_3)$ reads as
	\begin{equation*}
	\begin{aligned}
	\partial_t u_1 - d_1\Delta u_1 &= -2(u_1^2u_2 - u_3^2) &&=: f_1(u),\\
	\partial_t u_2 - d_2\Delta u_2 &= -u_1^2u_2 + u_3^2 &&=: f_2(u),\\
	\partial_t u_3 - d_3\Delta u_3 &= +2(u_1^2u_2 - u_3^2) &&=: f_3(u),
	\end{aligned}
	\end{equation*}
	with homogeneous Neumann boundary condition $\nabla u_i \cdot \nu = 0$ and initial data $u_i(\cdot,0) = u_{i,0}(\cdot)$. Here $d_1, d_2, d_3>0$ are positive diffusion coefficients. It is obvious that $f_i(u)$ satisfies the positivity preserving property \eqref{positive-preserve} and the growth condition \eqref{cubic-growth} in one dimension. Moreover, with $\mu_1 = \mu_2 = \mu_3 = 0$ we have
	\begin{equation*}
	\sum_{i=1}^{3}f_i(u)(\mu_i + \log u_i) = -(u_1^2u_2 - u_3^2)(\log{(u_1^2u_2)} - \log {(u_3^2)}) \leq 0
	\end{equation*}
	since the function $\log x$ is increasing. That means the entropy inequality \eqref{entropy-inequality} is satisfied.
\end{example}

\medskip
Condition \eqref{positive-preserve} has a simple interpretation: when the $i$-th concentration is zero then it cannot be consumed in the reaction. This assumption is sometimes called {\it quasi-positivity} and it helps to obtain the positivity of solutions provided initial data are positive (see e.g. \cite{Pao}). The entropy inequality \eqref{entropy-inequality} provides a control on solutions in $L\log L$-norm (see Lemma \ref{entropy-estimates}), and it is guaranteed in many physical systems, see e.g. chemical reaction networks satisfying a complex balanced condition in Section \ref{applications}. Loosely speaking, \eqref{entropy-inequality} means that the free energy of the corresponding system is dissipating. Condition \eqref{cubic-growth} is the only "real" restriction of systems under consideration in this paper. This restriction on the growth of nonlinearities is necessary for obtaining from $L\log L$-bound (implied by \eqref{entropy-inequality}) suitable {\it a priori} estimates, which in turn lead to global classical solutions. Possible extensions to higher orders or higher dimensions, for example $\mu = 2$ and $d\geq 3$, remains as an open problem.

\medskip
Systems of type \eqref{cubic-system} with the conditions \eqref{positive-preserve}, \eqref{entropy-inequality} 
%and \eqref{cubic-growth} arise typically from chemical reaction networks with detailed- or complex balance condition, and thus 
has been extensively investigated (see e.g. \cite{CDF14,CV09,DFPV07,FLS16,FL16,GH97,GV10,Rot87}).
We refer the reader to the review paper \cite{Pie-Survey} for a detailed discussion. It is worth mentioning that the global existence of classical solution to \eqref{cubic-system} is widely open in general.
For weaker notions of solutions, we refer the reader to \cite{Fis15} where the author showed the global existence of renormalised solutions under the conditions \eqref{positive-preserve} and \eqref{entropy-inequality}.

\medskip
To put our work into context, let us recall that in \cite{CDF14}, by using a duality method the authors have proved the global existence of classical solutions to \eqref{cubic-system} with quadratic nonlinearities, i.e. $|f_i(u)| \leq C(|u|^2 + 1)$, in one and two dimensions. When the dimension is three or higher, the classical solution is proved global provided the diffusion coefficients are closed to each other. We remark that, since the duality method is independent of dimensions, the arguments in \cite{CDF14} are not directly applicable to systems with cubic nonlinearities (even in one dimension).
It's also worth noting that the global classical solutions to quadratic systems in higher dimensions (without the assumption on diffusion coefficients) had remained open until the very recent preprint \cite{CGV17}, in which the authors solved the problem (even for slightly super quadratic - depending on the dimension - systems) by utilising the De Giorgi method. 

{On the other hand, it was proved in \cite{GV10} that  system \eqref{cubic-system} with conditions \eqref{positive-preserve}, \eqref{entropy-inequality} and \eqref{cubic-growth} has global classical solutions. However, the proof therein was based on the famous De Giorgi method, and apparently did not provide any bounds (w.r.t. time) on the classical solutions.}

\medskip
In this paper, with the help of a modified Gagliardo-Nirenberg's inequality and the regularity of the heat operator, we prove by simple arguments that system \eqref{cubic-system} with conditions \eqref{positive-preserve}, \eqref{entropy-inequality} and \eqref{cubic-growth} has a unique global classical solution. One advantage of our results is that the $L^{\infty}$-norm of the solution grows at most polynomially in time. This is usually called {\it slowly growing a-priori bounds}, see e.g. \cite{DF08,TV}. This result is very helpful, especially when the asymptotic behaviour of solutions in weaker norm, say $L^1$-norm, is already established (see Corollary \ref{cor:CB-system} for an application to chemical reaction networks). Note finally that similar ideas were also used in \cite{GH97,PSY17,SY15} to obtain global solutions for systems with quadratic nonlinearities in two dimensions.

\medskip
The main result of this paper is the following.
\begin{theorem}\label{theo:main}
	Let $\Omega\subset \mathbb R^d$ with $d=1,2$ be a bounded domain with smooth boundary $\partial\Omega$ in case $d = 2$ (e.g. $\partial\Omega$ is of class $ C^{2+\epsilon}$ with $\epsilon>0$).
	Assume the diffusion coefficients are positive, i.e. $d_i>0$ for all $i=1,\ldots, N$, and the nonlinearities are locally Lipschitz and satisfy \eqref{positive-preserve}, \eqref{entropy-inequality} and \eqref{cubic-growth}. Then for any initial data $(u_{i,0})\in L^{\infty}(\Omega)^N$, system \eqref{cubic-system} has a unique classical solution on $(0,\infty)$. Moreover, the $L^{\infty}$-norm of this solution grows at most polynomially in time, i.e. for all $T>0$
	\begin{equation*}
	\|u_i(t)\|_{L^{\infty}(\Omega)} \leq C_T \quad \text{ for all } \quad 0<t\leq T \text{ and all } i=1,\ldots, N,
	\end{equation*}
	in which $C_T$ is a constant depends at most polynomially w.r.t. $T$.
\end{theorem}

\begin{remark}
	The entropy condition \eqref{entropy-inequality} can in fact be weakened as
	\begin{equation*}\tag{$\mathbf E'$}
		\sum_{i=1}^{N}f_i(u)(\mu_i + \log u_i) \leq K_1\sum_{i=1}^{N}u_i + K_2 \quad \text{ for all } \quad u\in \mathbb R_+^N,
	\end{equation*}
	with positive constants $K_1, K_2>0$.
	See e.g. \cite{PSY17} for more details.
\end{remark}
\begin{remark}
	The $L^{\infty}(\Omega)$ initial data, which is chosen for the sake of simplicity, is certainly not optimal. For example, with more careful analysis, the initial data can be relaxed to $L^p(\Omega)$ for some $p>2$ (when one only wishes to have the bounds of solution after some time $t_0>0$).
\end{remark}

The rest of this paper is organised as follows: In the next section, we give the proof of Theorem \ref{theo:main}, and then its application to complex balanced systems is described in Section \ref{applications}.

\medskip
{\bf Notations}: For simplicity, we will use the following set of notations:
\begin{itemize}
	\item The norm in $L^{p}(\Omega)$ with $1\leq p \leq \infty$ is denoted by $\|\cdot\|_{L^p}$.
	\item For any $T>0$ we write $Q_T = \Omega\times (0,T)$ and $\|\cdot\|_{L^p(Q_T)}:= \|\cdot\|_{L^p(0,T;L^p(\Omega))}$ for $p\geq 1$.
	\item We will denote by $C_i$ a varying constant (possibly) depending on the domain $\Omega$, the diffusion coefficients, etc., but {\it independent of time $t$}.
\end{itemize}
\section{Proof of main results}
The local existence of classical solutions to \eqref{cubic-system} with locally Lipschitz nonlinearities is standard (see e.g. \cite{Ama85,Rot87}). Moreover, thanks to the positivity preserving property \eqref{positive-preserve} of the nonlinearities, the solution is positive as long as it exists, see e.g. \cite{Pao}.

To prove that the classical solution is global, we first show in Lemma \ref{entropy-estimates} a uniform-in-time bound for solution in $L^1$- and $L\log L$-norms. These bounds, in a combination with a modified Gagliardo-Nirenberg inequality in Lemma \ref{modified}, lead then to some $L^p$-integrability of solutions in Lemma \ref{L3} for suitable $p$. Finally this integrability, with the help from smoothing effect of the heat operator in Lemma \ref{heat-regularity}, implies the bound of solutions in $L^{\infty}$.

\begin{lemma}[Entropy estimate]\label{entropy-estimates}
	Assume that \eqref{entropy-inequality} holds. Then we have the following a priori estimate for any classical solution to \eqref{cubic-system},
	\begin{equation*}
	\sup_{t\geq 0} \|u_i(t)\|_{L^1} \leq C_0, \qquad \sup_{t\ge 0}\|u_i(t)\log u_i(t)\|_{L^1} \leq C_0
	\end{equation*}
	here $C_0$ is a constant independent of $t$.
\end{lemma}
\begin{proof}
	From the entropy inequality \eqref{entropy-inequality} we have
	\begin{equation}\label{e0}
	\begin{aligned}
	\frac{d}{dt}&\int_{\Omega}\sum_{i=1}^{N}(u_i(\mu_i + \log u_i) - u_i)dx\\
	&= \int_{\Omega}\sum_{i=1}^{N}(\mu_i + \log u_i) \partial_t u_i dx\\
	&= -\int_{\Omega}\sum_{i=1}^{N}d_i\frac{|\nabla u_i|^2}{u_i}dx +\int_{\Omega} \sum_{r=1}^{N}(\mu_i + \log u_i)f_i(u)dx\\
	&\leq -\int_{\Omega}\sum_{i=1}^{N}d_i\frac{|\nabla u_i|^2}{u_i}dx \leq 0.
	\end{aligned}
	\end{equation}
	Hence,
	\begin{equation*}
	\int_{\Omega}\sum_{i=1}^{N}(u_i(t)\log u_i(t) - u_i(t) + \mu_iu_i(t))dx \leq \int_{\Omega}\sum_{i=1}^{N}(u_{i,0}\log u_{i,0} - u_{i,0} + \mu_iu_{i,0})dx.
	\end{equation*}
	Denote by $C_1$ the right hand side, we rewrite this inequality as
	\begin{equation}\label{e1}
	\begin{aligned}
	\int_{\Omega}\sum_{i=1}^{N}(u_i(t)\log u_i(t) - u_i(t) + 1)dx &\leq C_1 + N|\Omega| - \int_{\Omega}\sum_{i=1}^{N}\mu_iu_i(t)dx\\
	&\leq C_1 + N|\Omega| + \max_{i=1,\ldots, N}|\mu_i|\int_{\Omega}\sum_{i=1}^{N}u_i(t)dx.
	\end{aligned}
	\end{equation}
	Using the inequality $x\log x - x + 1 \geq Lx - e^L + 1$ for all $L>0$ we have 
	\begin{equation*}
	\int_{\Omega}\sum_{i=1}^{N}(u_i(t)\log u_i(t) - u_i(t) + 1)dx \geq K\int_{\Omega}\sum_{i=1}^{N}u_i(t)dx - e^KN|\Omega| + N|\Omega|
	\end{equation*}
	with  $K = 2\max_{i=1,\ldots, N}|\mu_i|$. Therefore we obtain the estimate
	\begin{equation*}
	\int_{\Omega}\sum_{i=1}^{N}u_i(t)dx \leq \frac{1}{K}(C_1 + e^KN|\Omega|)
	\end{equation*}
	which, together with the positivity of the solution, leads to the uniform in tim $L^1$-bound. The bound of $\|u_i(t)\log u_i(t)\|_{L^1}$ follows immediately from \eqref{e1} and $x\log x \geq x - 1$ for all $x\geq 0$.
\end{proof}
\begin{lemma}[A modified Gagliardo-Nirenberg inequality]\label{modified}
	For any $\varepsilon>0$ there exists $C_{\varepsilon}>0$ such that for all $f\in H^1(\Omega)$
	\begin{equation*}
	\|f\|_{L^4}^4 \leq \varepsilon\|f\|_{H^1(\Omega)}^2\|f\log |f|\|_{L^1}^2 + C_{\varepsilon}\|f\|_{L^1} \quad \text{ when } \quad d = 1,
	\end{equation*}
	and
	\begin{equation*}
	\|f\|_{L^3}^3 \leq \varepsilon\|f\|_{H^1(\Omega)}^2\|f\log |f|\|_{L^1} + C_{\varepsilon}\|f\|_{L^1} \quad \text{ when } \quad  d= 2.
	\end{equation*}		
\end{lemma}
\begin{proof}
	The proof follows from the ideas in \cite{BWT94} where the authors obtained a similar version (with estimates for $L^3$-norm) in two dimensions. We only give a proof in case $d=1$ since the proof in case $d=2$ is similar.
	
	Fix a constant $N>1$. Define a function $\chi: \mathbb R \to \mathbb R$ as $\chi(s) = 0$ if $|s| \leq N$, $\chi(s) = 2(|s| - N)$ when $N<|s| \leq 2N$ and $\chi(s) = |s|$ when $|s| > 2N$. In this proof we use
	\[
	\Omega\{|f| \geq N\}:= \{x\in\Omega: |f(x)| \geq N\}.
	\]
	First we write
	\begin{equation}\label{e2}
	\|f\|_{L^4}^4 \leq \left(\|\chi(f)\|_{L^4} + \||f| - \chi(f)\|_{L^4}\right)^4 \leq 8(\|\chi(f)\|_{L^4}^4 + \||f| - \chi(f)\|_{L^4}^4)
	\end{equation}
	and then estimate each term separately. It is easy to see that
	\begin{equation}\label{e3}
	\||f| - \chi(f)\|_{L^4}^4 = \int_{\Omega}||f| - \chi(f)|^4dx \leq (2N)^3\int_{\Omega\{|f|\leq 2N\}}|f|dx \leq (2N)^3\|f\|_{L^1}.
	\end{equation}
	Concerning the other term, we use the usual Gagliardo-Nirenberg inequality
	\begin{equation}\label{e4}
	\|\chi(f)\|_{L^4}^4 \leq C_4\|\chi(f)\|_{H^1(\Omega)}^2\|\chi(f)\|_{L^1}^2
	\end{equation}
	for some constant $C_4>0$. On the one hand
	\begin{equation}\label{e5}
	\|\chi(f)\|_{H^1(\Omega)}^2 = \|\chi'(f)\nabla f\|_{L^2}^2 + \|\chi(f)\|_{L^2}^2 \leq 4\|f\|_{H^1(\Omega)}^2,
	\end{equation}
	and on the other hand
	\begin{equation}\label{e6}
	\|\chi(f)\|_{L^1}^2 \leq \left(\int_{\Omega\{|f|\geq N\}}|f|dx \right)^2 \leq  (\log N)^{-2}\|f\log|f|\|_{L^1}^2.
	\end{equation}
	By combining \eqref{e2}--\eqref{e6} we obtain
	\begin{equation}
	\|f\|_{L^4}^4 \leq 16C_4(\log N)^{-2}\|f\|_{H^1(\Omega)}^2\|f\log|f|\|_{L^1}^2 + 4(2N)^3\|f\|_{L^1}.
	\end{equation}
	At this point we can choose $N$ to be large enough to obtain the desired inequality.
\end{proof}
\begin{lemma}\label{L3}
	Assume that \eqref{positive-preserve}, \eqref{entropy-inequality} and \eqref{cubic-growth} hold. Then for any $T>0$, we have for all $i=1,\ldots, N$
	\begin{equation}\label{L4}
	\|u_i\|_{L^4(Q_T)} \leq C_T \text{ when } d=1,
	\end{equation}
	and
	\begin{equation}\label{L4epsilon}
	\|u_i\|_{L^{4-\epsilon}(Q_T)} \leq C_T \text{ when } d=2,
	\end{equation}	
	with $\epsilon>0$ arbitrary, where $C_T$ is a constant grows at most polynomially w.r.t. $T$.
\end{lemma}
\begin{proof}
	Recall that $C_i$ denotes a various constant depending on the domain $\Omega$, the diffusion coefficients, the constant $K$ in \eqref{cubic-growth}, and constant $C_0$ in Lemma \ref{entropy-estimates}, but {\it independent} of time $t$. Note that all constants $C_i$ can be explicitly computed.
	
	We first prove \eqref{L4}. Multiplying the equation
	\begin{equation*}
	\partial_t u_i - d_i\Delta u_i = f_i(u)
	\end{equation*}
	by $u_i$ in $L^2(\Omega)$ we have
	\begin{equation}\label{e7}
	\begin{aligned}
	\frac 12 \frac{d}{dt}\|u_i\|_{L^2}^2 + d_i\|\nabla u_i\|_{L^2}^2 &= \int_{\Omega}f_i(u)u_idx\\
	&\leq K\int_{\Omega}(|u|^3 + 1)|u_i|dx\\
	&\leq 3K\int_{\Omega}\sum_{j=1}^{N}|u_i||u_j|^3dx + KC_0\quad (\text{by Lemma \ref{entropy-estimates}})\\
	&\leq 3K\sum_{j=1}^{N}(\|u_j\|_{L^4}^4 + \|u_i\|_{L^4}^4) + KC_0 \quad (\text{H\"older's inequality})\\
	&\leq 3NK\sum_{j=1}^{N}\|u_j\|_{L^4}^4+ KC_0.
	\end{aligned}
	\end{equation}
	Hence, by summing over $i=1,\ldots, N$ and using Lemma \ref{modified}, it follows that
	\begin{equation*}
	\begin{aligned}
	\frac{d}{dt}&\sum_{i=1}^{N}\|u_i\|_{L^2}^2 + 2\sum_{i=1}^{N}d_i\|\nabla u_i\|_{L^2}^2\\
	& \leq 3N^2K\sum_{i=1}^{N}\|u_i\|_{L^4}^4 + KNC_0\\
	& \leq 3N^2K\sum_{i=1}^{N}\left(\varepsilon\|u_i\|_{H^1(\Omega)}^2 \|u_i\log u_i\|_{L^1}^2  + C_{\varepsilon}\|u_i\|_{L^1}\right)\\
	&\leq 3\varepsilon N^3KC_0^2\sum_{i=1}^{N}\|u_i\|_{H^1(\Omega)}^2 + 3N^3KC_{\varepsilon}C_0\quad (\text{by Lemma \ref{entropy-estimates}}).
	\end{aligned}
	\end{equation*}
	Choosing $\varepsilon$ small enough we get
	\begin{equation*}
	\frac{d}{dt}\sum_{i=1}^{N}\|u_i\|_{L^2}^2 + 2\sum_{i=1}^{N}d_i\|\nabla u_i\|_{L^2}^2 \leq \sum_{i=1}^{N}d_i\|u_i\|_{H^1(\Omega)}^2 + C_6
	\end{equation*}
	for some $C_6>0$. By adding $2\sum_{i=1}^{N}d_i\|u_i\|_{L^2}^2$ to both sides and using the one dimensional embedding inequality $\|u_i\|_{L^{\infty}}^2 \leq C_7\|u_i\|_{H^1(\Omega)}^2$, it follows that
	\begin{equation*}
	\begin{aligned}
	\frac{d}{dt}\sum_{i=1}^{N}\|u_i\|_{L^2}^2 + C_7\sum_{i=1}^{N}d_i\|u_i\|_{L^{\infty}}^2 &\leq C_6 + 2\sum_{i=1}^{N}d_i\|u_i\|_{L^2}^2\\
	&\leq C_6 + 2\sum_{i=1}^{N}d_i\|u_i\|_{L^1}\|u_i\|_{L^{\infty}} \quad (\text{interpolation})\\
	&\leq C_6 + C_8\sum_{i=1}^{N}\|u_i\|_{L^1}^2 + \frac{C_7}{2}\sum_{i=1}^{N}d_i\|u_i\|_{L^\infty}^2\\
	&\leq C_9 +\frac{C_7}{2}\sum_{i=1}^{N}d_i\|u_i\|_{L^\infty}^2 \quad (\text{by } \|u_i\|_{L^1} \leq C_0).
	\end{aligned}
	\end{equation*}
	Thus, using one more $\|u_i\|_{L^{\infty}}^2 \geq \frac{1}{|\Omega|}\|u_i\|_{L^2}^2$, it leads to
	\begin{equation*}
	\frac{d}{dt}\sum_{i=1}^{N}\|u_i\|_{L^2}^2 + C_{10}\sum_{i=1}^{N}\|u_i\|_{L^2}^2 + C_{11}\sum_{i=1}^{N}\|u_i\|_{L^{\infty}}^2\leq C_9
	\end{equation*}
	and consequently an integration on $(0,T)$ gives
	\begin{equation*}
	\sum_{i=1}^{N}\|u_i(t)\|_{L^2}^2 + C_{11}\sum_{i=1}^{N}\int_{0}^{T}\|u_i(\tau)\|_{L^\infty}^2d\tau \leq \sum_{i=1}^{N}\|u_{i,0}\|_{L^2}^2 + \frac{C_9}{C_{10}}T\quad \text{ for all } \quad t>0.
	\end{equation*}
	From this we have $u_i\in L^{\infty}(0,T;L^2(\Omega))$ and $u_i\in L^2(0,T;L^{\infty}(\Omega))$ for all $i=1,\ldots, N$. Finally, by an interpolation 
	\begin{equation*}
	L^{\infty}(0,T;L^2(\Omega))\cap L^2(0,T;L^{\infty}(\Omega))\hookrightarrow L^4(Q_T)
	\end{equation*}
	we obtain \eqref{L4}.
	
	\medskip	
	{Consider now the case $d=2$. With computations similar to \eqref{e7} we get
		\begin{equation*}
		\frac{d}{dt}\sum_{i=1}^{N}\|u_i\|_{L^2}^2 + 2\sum_{i=1}^{N}d_i\|\nabla u_i\|_{L^2}^2 \leq C_{12}\sum_{i=1}^{N}\|u_i\|_{L^3}^3 + C_{13}.
		\end{equation*}
		Applying Lemma \ref{modified} with $d=2$ to $\|u_i\|_{L^3}^3$, and using the bounds $\sup\limits_{t\geq 0}\|u_i(t)\|_{L^1}\leq C_0$ and $\sup\limits_{t\geq 0}\|u_i(t)\log u_i(t)\|_{L^1}\leq C_0$ in Lemma \ref{entropy-estimates} lead to
		\begin{equation*}
		\frac{d}{dt}\sum_{i=1}^{N}\|u_i\|_{L^2}^2 + 2\sum_{i=1}^{N}d_i\|\nabla u_i\|_{L^2}^2 \leq \sum_{i=1}^{N}d_i\|u_i\|_{H^1(\Omega)}^2 + C_{14}.
		\end{equation*}
		By adding both sides with $\sum_{i=1}^{N}d_i\|u_i\|_{L^2}^2$ and using the two-dimensional embedding $H^1(\Omega)\hookrightarrow L^p(\Omega)$ for any $2\leq p <+\infty$, we end up with
		\begin{equation}\label{e8}
		\frac{d}{dt}\sum_{i=1}^{N}\|u_i\|_{L^2}^2 + \frac 12\sum_{i=1}^{N}d_i\|u_i\|_{L^2}^2 + C_{15}\sum_{i=1}^{N}\|u_i\|_{L^p}^2 \leq \sum_{i=1}^{N}d_i\|u_i\|_{L^2}^2 + C_{14}.
		\end{equation}
		Interpolation inequality and Young's inequality give
		\begin{equation*}
		\|u_i\|_{L^2}^2 \leq \|u_i\|_{L^p}^{\frac{p}{p-2}}\|u_i\|_{L^1}^{\frac{p-2}{p-1}} \leq \frac{d_i}{4}\|u_i\|_{L^p}^2 + C_{16}\|u_i\|_{L^1}^2.
		\end{equation*}
		Inserting this into \eqref{e8} and using the $L^1$-bound of $u_i$ we finally obtain, after integrating on $(0,T)$,
		\begin{equation*}
		\sum_{i=1}^{N}\|u_i(t)\|_{L^2}^2 + C_{17}\sum_{i=1}^{N}\int_{0}^{T}\|u_i(t)\|_{L^p}^2dt \leq \sum_{i=1}^{N}\|u_{i,0}\|_{L^2}^2 + C_{18}T\quad \text{ for all } \quad t>0.
		\end{equation*}
		Hence $u_i\in L^{\infty}(0,T;L^2(\Omega))\cap L^2(0,T;L^p(\Omega))$. Therefore the interpolation
		\begin{equation*}
		L^{\infty}(0,T;L^2(\Omega))\cap L^2(0,T;L^p(\Omega)) \hookrightarrow L^{4-\frac 4p}(Q_T)
		\end{equation*}
		and the fact that $1\leq p <+\infty$ is arbitrary give us the desired estimate \eqref{L4epsilon}.
	}
\end{proof}

We need the following the regularity of solutions to a heat equation with homogeneous Neumann boundary condition. The proof can be found in \cite[Lemma 3.3]{CDF14}.
\begin{lemma}[Regularity of heat kernel]\label{heat-regularity}\cite{CDF14}
	Let $\Omega\subset \mathbb R^d$ be a bounded domain with smooth boundary $\partial\Omega$ (e.g. $\partial\Omega$ is of class $C^{2+\epsilon}$ with $\epsilon>0$). Consider the heat equation
	\begin{equation*}
	\partial_ty - \delta \Delta y = f \quad \text { in } \quad \Omega\times (0,T)
	\end{equation*}
	where $\delta>0$, with homogeneous Neumann boundary condition $\nabla y \cdot \nu = 0$ on $\partial\Omega$, and initial data $y(\cdot,0) = y_0\in L^{\infty}(\Omega)$. Assume that  right hand side $f\in L^p(Q_T)$ with $p\geq 1$.
	
	\begin{itemize}
		\item[(i)] If $p< (d+2)/2$ then 
		\begin{equation*}
		\|y\|_{L^s(Q_T)} \leq C_T \quad \text{ for all } \quad 1\leq s < \frac{p(d+2)}{d + 2 - 2p}.
		\end{equation*}
		\item[(ii)] If $p\geq (d+2)/2$ then
		\begin{equation*}
		\|y\|_{L^{r}(Q_T)} \leq C_T \quad \text{ for all } \quad 1\leq r < \infty.
		\end{equation*}
	\end{itemize}
	Here $C_T$ is a constant depending on domain $\Omega$, the integrability $p$ and the diffusion coefficient $\delta$, and especially depending {\normalfont at most polynomially} on $T$.
\end{lemma}
\begin{remark}
	Note that the regularity of solution to heat equation presented in Lemma \ref{heat-regularity} is classical (see e.g. \cite{Ladyzhenskaya}). The novelty of the lemma is to provide the bound $\|y\|_{L^s(Q_T)} \leq C_T$ in which $C_T$ is a constant depending {\it at most polynomially} in $T$.
\end{remark}

\medskip
We are now ready to give a proof of Theorem \ref{theo:main}.
\begin{proof}
	We will prove that for $i=1, \ldots, N$ it holds $\|u_i(t)\|_{L^{\infty}} < \infty$ for all $t>0$ and thus confirms the global existence of classical solution.
	
	\medskip
	In this proof, we will always denote by $C_T$ a constant depending \textit{polynomially} in $T>0$. 
	
	If $d=1$ and $\mu=3$ then it follows from $\|u_i\|_{L^4(Q_T)} \leq C_T$ in Lemma \ref{L3} and \eqref{cubic-growth} that
	\begin{equation*}
	\|f_i(u)\|_{L^{4/3}(Q_T)} \leq C_T
	\end{equation*}
	By applying Lemma \ref{heat-regularity} (i) to 
	\begin{equation*}
	\partial_t u_i - d_i\Delta u_i = f_i(u), \qquad \nabla u_i\cdot\nu = 0,
	\end{equation*}
	with $d = 1$ and $p = 4/3$ we have for all $i=1,\ldots, N$
	\begin{equation*}
	\|u_i\|_{L^s(Q_T)} \leq C_T \quad \text{ for all } \quad 1\leq s < 12.
	\end{equation*}
	Using again the cubic growth \eqref{cubic-growth} we obtain for $i=1,\ldots, N$,
	\begin{equation*}
	\|f_i(u)\|_{L^{s'}(Q_T)} \leq C_T \quad \text{ for all } \quad 1\leq s' < 4.
	\end{equation*}
	Thus, we apply Lemma \ref{heat-regularity} (ii) to get
	\begin{equation*}
	\|u_i\|_{L^r(Q_T)} \leq C_T \quad \text{ for all } \quad 1\leq r<+\infty.
	\end{equation*}
	Hence, $\|f_i(u)\|_{L^r(Q_T)} \leq C_T$ for all $1\leq r<+\infty$. One more bootstrap gives us the desired result
	\begin{equation*}
	\|u_i(t)\|_{L^{\infty}} \leq C_T \quad \text{ for all } 0<t\leq T.
	\end{equation*}
	
	\medskip
	If $d=2$ and $\mu=2$ then we have 
	\begin{equation*}
	\|f_i(u)\|_{L^{2-\epsilon}}(Q_T) \leq C_T
	\end{equation*}
	due to \eqref{cubic-growth} and $\|u_i\|_{L^{4-\epsilon}}(Q_T) \leq C$ in Lemma \ref{L3}. Applying Lemma \ref{heat-regularity} (i) to $\partial_t u_i - d_i\Delta u_i = f_i(u)$ leads to
	\begin{equation*}
	\|u_i\|_{L^s(Q_T)} \leq C_T \quad \text{ for all } \quad 1\leq s<\frac{4-2\epsilon}{\epsilon}.
	\end{equation*}
	Since $\epsilon>0$ is arbitrary, we obtain in fact $\|u_i\|_{L^s(Q_T)} \leq C_T$ for all $1\leq s<+\infty$. Now Lemma \ref{heat-regularity} (ii) is applicable and finally provides the estimate $\|u_i(t)\|_{L^\infty} \leq C_T$ for all $0<t\leq T$.
	
	The uniqueness of classical solution follows immediately from the $L^{\infty}$-bound and the fact that the nonlinearities $f_i(u)$ are locally Lipschitz.
\end{proof}

\begin{remark}
	From Lemma \ref{L3} we obtain in particular $\|u_i\|_{L^{\infty}(0,T;L^2(\Omega))}\leq C_T$ by using entropy bound in Lemma \ref{entropy-estimates}, the restriction on dimension $d=1,2$ and the growth \eqref{cubic-growth}. If the bound $L^{\infty}(0,T;L^2(\Omega))$ can be obtained from some other way (without using \eqref{cubic-growth}) then we can in fact improve the results of Theorem \ref{theo:main} to $\mu < 5$ for $d=1$ and $\mu < 3$ for $d=2$. The interested reader is referred to \cite[Section 5]{Tan17} for more details.
\end{remark}

\section{Applications to complex balanced systems}\label{applications}
One advantage of Theorem \ref{theo:main} is that it does not only provides the global existence of a classical solutions, but also gives a control on the growth (w.r.t to time) of $L^{\infty}$-norm of the solution. This becomes very helpful in the situation that we describe in the following.

\medskip
A direct application of Theorem \ref{theo:main} can be seen in chemical reaction network theory. Consider $N$ chemical substances $S_1, \ldots, S_N$ reacting in a network consisting of $R$ reactions, in which the $r$-th reaction is of the form
\begin{equation*}
\begin{tikzpicture}[baseline=(current  bounding  box.center)]
\node (a) {$y_{r,1}S_1 + \ldots + y_{r,N}S_N$} node (b) at (5,0) {$y_{r,1}'S_1 + \ldots + y_{r,N}'S_N$.};
\draw[arrows=->] (a.east) -- node [above] {\scalebox{.8}[.8]{$k_r$}} (b.west);
\end{tikzpicture}
\end{equation*}
Here $k_r>0$ is the reaction constant rate, $y_{r,i}, y_{r,i}' \in \{0\}\cup [1,\infty)$ are stoichiometric coefficients. By using the notation $y_r = (y_{r,1}, \ldots, y_{r,N})$ and $y_{r}' = (y_{r,1}, \ldots, y_{r,N})$ we can rewrite the $r$-th reaction as
\begin{equation*}
\begin{tikzpicture}[baseline=(current  bounding  box.center)]
\node (a) {$y_r$} node (b) at (2,0) {$y_r'$};
\draw[arrows=->] (a.east) -- node [above] {\scalebox{.8}[.8]{$k_r$}} (b.west);
\end{tikzpicture}
\end{equation*}
in which we call $y_r$ a reactant and $y_r'$ a production, and both of them are named {\it complex}. Denote by $\mathfrak C = \{y_r, y_r'\}_{r=1,\ldots, R}$ the set of all complexes. Note that each complex $y\in \mathfrak C$ can be both a reactant and a production (in possibly different reactions). Assume now that the following reaction network is taken place is a bounded domain $\Omega\subset \mathbb R^d$ and each substance $S_i$ diffuses with a constant rate $d_i>0$. Hence, by applying the {\it law of mass action} we obtain the reaction-diffusion system for the concentrations $u_1, \ldots, u_N$ of $S_1, \ldots, S_N$ respectively,
\begin{equation}\label{mass-action-system}
\begin{aligned}
\partial_t u_i - d_i\Delta u_i &= f_i(u):= \sum_{r=1}^R\left(k_r(y_{r,i} - y_{r,i}')\prod\limits_{i=1}^{N}u_i^{y_{r,i}}\right), && x\in\Omega, &&& t>0,\\
\nabla u_i \cdot \nu &= 0, && x\in\partial\Omega, &&& t>0,\\
u_i(x,0) &= u_{i,0}(x), && x\in\Omega
\end{aligned}
\end{equation}
where $\nu$ is the outward normal on $\partial\Omega$. System \eqref{mass-action-system} is called {\it complex balanced} if there exists a strictly positive equilibrium $u_{\infty} = (u_{1,\infty}, \ldots, u_{N,\infty})\in \mathbb R^N_{+}$ such that at $u_{\infty}$ the total in-flow and total out-flow at any complex $y\in \mathfrak C$ are balanced, i.e.
\begin{equation}\label{CB-condition}
\sum_{\{r:\; y_r = y\}}\left(k_r\prod\limits_{i=1}^{N}u_{i,\infty}^{y_{r,i}}\right) = \sum_{\{r:\; y_r' = y\}}\left(k_r\prod\limits_{i=1}^{N}u_{i,\infty}^{y_{r,i}}\right).
\end{equation}
For more details concerning complex balanced systems, the interested reader is referred to \cite{DFT16,Fei1,Fei2}. Note that \eqref{mass-action-system} can also have {\it boundary equilibrium} $u^*$, that is $u^*$ satisfies \eqref{CB-condition} and $u^*\in \partial\mathbb R^N_+$. The convergence to equilibrium for systems of type \eqref{mass-action-system} was extensively studied recently, see e.g. \cite{DFT16,FT17,MHM14,PSZ16} and references therein. In particular, it was proved in \cite{FT17} that if \eqref{mass-action-system} is complex balanced and has no boundary equilibria, then any renormalised solution (see \cite{Fis15}) converges in $L^1$-norm exponentially to the unique strictly positive complex balanced equilibrium $u_{\infty}$. Thus, by 
applying Theorem \ref{theo:main} we can show that the classical solution (in cases globally exists) in fact converges to equilibrium in $L^{\infty}$-norm exponentially.
\begin{corollary}\label{cor:CB-system}
	Let $\Omega\subset \mathbb R^d$ be a bounded domain with smooth enough boundary $\partial\Omega$ (e.g. $\partial\Omega$ is of class $C^{2+\epsilon}$ for $\epsilon>0$). 
	Assume either
	\begin{equation*}
	d = 1 \quad \text{ and } \quad \max_{r=1,\ldots, R}|y_r| \leq 3
	\end{equation*}
	or
	\begin{equation*}
	d = 2 \quad \text{ and } \quad \max_{r=1,\ldots, R}|y_r| \leq 2,
	\end{equation*}	
	recalling that $d$ is the spatial dimension and $|y| = |y_1| + \ldots + |y_N|$ for $y\in\mathbb R^N$. Moreover, assume that system \eqref{mass-action-system} is complex balanced and has no boundary equilibria. Then for any nonnegative initial data $u_0 \in L^{\infty}(\Omega)$, the system \eqref{mass-action-system} has a unique global classical solution which converges exponentially in $L^{\infty}$-norm to the corresponding complex balanced equilibrium $u_{\infty}$, i.e.
	\begin{equation*}
	\sum_{i=1}^{N}\|u_i(t) - u_{i,\infty}\|_{L^{\infty}(\Omega)} \leq Ce^{-\lambda t} \quad \text{ for all } \quad t>0,
	\end{equation*}
	where $C$ and $\lambda$ are positive constants.
\end{corollary}
\begin{remark}
	In the recent work \cite{Fis17} Fischer proved the weak-strong uniqueness result for \eqref{mass-action-system}, that is a renormalised solution to \eqref{mass-action-system} identifies with the classical solution in the time interval where the latter exists. By using the global existence and uniqueness of classical solution in Corollary \ref{cor:CB-system}, we consequently obtain the uniqueness of renormalised solution to \eqref{mass-action-system}.
\end{remark}
\begin{proof} 
	Since system \eqref{mass-action-system} is complex balanced and has no boundary equilibria, it follows from \cite{FT17} that any renormalised solution converges exponentially to equilibrium $u_{\infty}$ in $L^1$-norm, i.e.
	\begin{equation*}
	\sum_{i=1}^{N}\|u_i(t) - u_{i,\infty}\|_{L^1} \leq Ce^{-\gamma t}
	\end{equation*}
	for $C, \gamma >0$. From Theorem \ref{theo:main} and the growth conditions of nonlinearities we have $\|u_i(T)\|_{L^{\infty}}\leq C_T$. Now for each $1<p<\infty$, by using an interpolation estimate we have
	\begin{equation*}
	\|u_i(T) - u_{i,\infty}\|_{L^p} \leq (\|u_i(T)\|_{L^\infty} + u_{i,\infty})^{\theta}\|u_i(T) - u_{i,\infty}\|_{L^1}^{1-\theta} \leq C_T^{\theta}e^{-(1-\theta)\gamma T} \leq Ce^{-\theta' t}
	\end{equation*}
	for some $\theta \in (0,1)$ and $0 < \theta' <(1-\theta)\gamma$, thus we get the exponential convergence in $L^p$-norm for all $1<p<\infty$. To obtain the convergence in $L^{\infty}$-norm, we first note that due to $\|u_i\|_{L^{\infty}(Q_T)} \leq C_T$ we have $\|f_i(u)\|_{L^{\infty}(Q_T)} \leq C_T$ and consequently
	\begin{equation*}
	\partial_tu_i - d_i\Delta u_i = f_i(u) \in L^\infty(Q_T),
	\end{equation*}
	thus it follows from \cite{Lam87} in particular that $\|\partial_tu_i\|_{L^2(Q_T)} \leq C_T$. Hence
	\begin{equation*}
	\begin{aligned}
	\|f_i(u)\|_{H^1(0,T;L^2(\Omega))} &\leq \|\partial_tf_i(u)\|_{L^2(Q_T)} + \|f_i(u)\|_{L^2(Q_T)}\\
	&\leq C\sum_{i=1}^{N}\|u_i\|_{L^{\infty}(Q_T)}^{\alpha}\sum_{i=1}^{N}\|\partial_tu_i\|_{L^2(Q_T)} + \|f_i(u)\|_{L^2(Q_T)}\\
	&\leq C_T
	\end{aligned}
	\end{equation*}
	where $\alpha>0$ is some fixed constant depending on the growth of the nonlinearity $f_i(u)$. On the other hand, due to the smoothing effect of the heat operator, for any small $0<\tau$ we have $u_i(\cdot,\tau)\in H^1(\Omega)$. Now consider the equation $\partial_tu_i - d_i\Delta u_i = f_i(u)$
	with initial time at $\tau$ and initial data $u_i(\cdot,\tau)$ in $H^1(\Omega)$ we have the following estimate (see e.g. \cite[Theorem 7.1.5]{Evans})
	\begin{equation*}
	\|u_i(T)\|_{H^2(\Omega)} \leq C(\|f_i(u)\|_{H^1(0,T;L^2(\Omega))} + \|u_i(\tau)\|_{H^1(\Omega)}) \quad \text{ for all } T \geq \tau.
	\end{equation*}

	Therefore, using the usual Gagliardo-Nirenberg inequality (both in the cases $d=1$ and $d=2$) we obtain for some $\theta\in (0,1)$ that
	\begin{equation*}
	\begin{aligned}
	\sum_{i=1}^{N}\|u_i(T) - u_{i,\infty}\|_{L^{\infty}} &\leq \sum_{i=1}^{N}\|u_i(T) - u_{i,\infty}\|_{H^2(\Omega)}^{\theta}\|u_i(T) - u_{i,\infty}\|_{L^1}^{1-\theta}\\
	&\leq \sum_{i=1}^{N}C_T^{\theta}Ce^{-\gamma(1-\theta) T} \leq Ce^{-\lambda T} \quad \text{ for all } T\geq \tau,
	\end{aligned}
	\end{equation*}
	for some $0< \lambda < \gamma(1-\theta)$, due to the fact that $C_T$ grows at most polynomially in $T$. Combining this with $u_i\in L^{\infty}(0,\tau;L^{\infty}(\Omega))$ we can finish the proof of Corollary \ref{cor:CB-system}.
\end{proof}

\vskip2mm

\par{\bf Acknowledgements:} The author would like to thank Prof. Laurent Desvillettes and Prof. Klemens Fellner for fruitful discussion, which leads to this work. This work is partially supported by International Training Program IGDK 1754 and NAWI Graz.

          \end{document}